\documentclass[12pt]{amsart}
\usepackage{geometry} 

\usepackage{amssymb}
\usepackage{amsthm}
\usepackage[usenames]{color}

\def\rouge{ \textcolor{red} }

 \newtheorem{theorem}{Theorem}[section]
\newtheorem{lemma}[theorem]{Lemma}
\newtheorem{proposition}[theorem]{Proposition}

\theoremstyle{definition}
\newtheorem{example}[theorem]{Example}
\newtheorem{remark}[theorem]{Remark}

\newtheorem{definition}[theorem]{Definition}

\numberwithin{equation}{section}

\def\h{{\bf H}}
\def\u#1{{\bf u}_{#1}}
\def\a{a}
\def\b{b}
\def\c{c}
\def\d{d}
\def\S{{\mathcal S}}
\def\t#1{{\bf t}_{#1}}

\title [Schubert and  $k$-Schur]{Schubert Polynomials and $k$-Schur functions} 
\author{Carolina Benedetti and Nantel Bergeron}
\address[Nantel Bergeron]
 {Department of Mathematics and Statistics\\York University\\
 Toronto, Ontario M3J 1P3\\
 CANADA}
 \email{bergeron@mathstat.yorku.ca}
 \urladdr{http://www.math.yorku.ca/bergeron}

\address[Carolina Benedetti]
 {Department of Mathematics and Statistics\\York University\\
 Toronto, Ontario M3J 1P3\\
 CANADA}
 \email{carobene@mathstat.yorku.ca}

\thanks{N. Bergeron is supported in part by NSERC}

\begin{document}

\begin{abstract}The main purpose of this paper is to show that the multiplication of a Schubert polynomial of finite type $A$ by a Schur function, which we refer to as Schubert vs. Schur problem, can be understood from the multiplication in the space of dual $k$-Schur functions. Using earlier work by the second author, we encode both problems by means of quasisymmetric functions. On the Schubert vs. Schur side, we study the poset given by the Bergeron-Sottile's $r$-Bruhat order, along with certain operators associated to this order. On the other side, we connect this poset with a graph on dual $k$-Schur functions given by studying the affine grassmannian order of  Lam-Lapointe-Morse-Shimozono. Also, we define operators associated to the graph on dual $k$-Schur functions which are analogous to the ones given for the Schubert vs. Schur problem.
\end{abstract}

\maketitle

\section{Introduction}\label{intro}
A fundamental problem in algebraic combinatorics is to find combinatorial rules for certain properties of a given combinatorial Hopf algebra. The problem of providing a combinatorial rule for the structure constants of a particular basis is an instance of this situation. The classical example is the Littlewood-Richardson rule which describes the multiplication and comultiplication of Schur functions within the space of symmetric functions. These constants are known to be positive from geometry since they describe intersections of grassmannian varieties, or from representation theory where they count the multiplicity of irreducible characters in certain induced representations.
Although this shows that, in theory, these constants are non-negative integers, the theory is not enough to specify whether they are non-zero, or how big they are. However, theLittlewood-Richarson rule does ~\cite{macdonald, knutsontao}  and
it describes each constant as the cardinality of a constructed set of objects or points.

Providing a rule for this kind of problems is in general very hard and many such problems are still unsolved. In particular, this paper will consider two of these problems which are closely related: the multiplication of Schubert polynomials, and the multiplication and comultiplication of $k$-Schur functions. Both contain as a particular case the so-called Gromov-Witten invariants. Let us give some background about each one of these problems.

Schubert polynomials are known to multiply positively since their structure constants enumerate flags in suitable triple intersections of Schubert varieties. However, there is no positive combinatorial rule to construct these constants in general. Nevertheless, since Schur polynomials correspond to grassmannian varieties which are a special class of flag varieties, we have that the Littlewood-Richardson rule is a special case of this particular problem. Even if we consider a slightly  larger class of Schubert polynomials, namely, multiplication of a Schubert polynomial by a Schur function, we find that for several years there was no solution for finding a positive rule for these structure constants. Fortunately, in ~\cite{bsottileschub} new identities were deduced, more tools were developed and the use of techniques along the way of ~\cite{bsottilemonoid,bsottilehopf,bsottileskew,bmws, A_JAMS} gave as a result a combinatorial rule for this problem~\cite{assafbsottile}, which we will refer later as Schubert vs. Schur. Also in~\cite{assafbsottile}, using the work of~\cite{BMP}, we deduce, independently of ~\cite{Buch}, a combinatorial proof that the Gromov-Witten invariants are positive.

Let us turn our attention now to $k$-Schur functions and their duals. These functions were first defined in ~\cite{LLM03} in order to study Macdonald polynomials but they soon turned out to be much more interesting due to their connection to different mathematical constructions. There are at least six different definitions of $k$-Schur functions and it is conjectural that they are equivalent. In~\cite{Lam}, one definition is shown to be related to the homology of the affine grassmannian of the affine coxeter group $\tilde{A}_{k+1}$. More precisely, the $k$-Schur functions are shown to be the Schubert polynomials for the affine grassmannian and, as such, the structure constants of their multiplication must be positive integers.
The space of $k$-Schur functions span a graded Hopf algebra, and its graded dual describes the cohomology of the affine grassmannian.
Thus, the comultiplication structure is also given by positive integer constants.
Also, the structure constants of $k$-Schur functions include, as a special case, the structure of the small quantum cohomology and in particular, as mentioned above, the Gromov-Witten invariants~\cite{lapointemorse}. 

In a series of two papers we plan to give a positive rule (along the lines of ~\cite{assafbsottile}) for the multiplication of dual $k$-Schur with a Schur function and relate this to the Schubert vs Schur problem. This is done by an in-depth study of the affine strong Bruhat graph. In order to achieve this we need to adapt the tools we have in~\cite{bsottilemonoid,bsottilehopf,bsottileskew,bmws,assafbsottile} and create new ones. To give an outline of how this will be done, we set up some notation. Partitions will be denoted by $\lambda,\mu,\nu$ and $u,v,w$ will denote affine grassmannian permutations. The general plan is as follows.

\noindent{\bf (I)} We study the strong Bruhat graph restricted to affine grassmannian permutations (see~\cite{LLMS}). Given two such permutations $u,v$ let $K_{[u,v]}$ be the quasisymmetric function associated to them, which is constructed following techniques in~\cite{bmws}. 
The coefficient $d_{u,\lambda}^v$ of a Schur  function $S_\lambda$ in $K_{[u,v]}$ is the same as the coefficient of the dual $k$-Schur ${S}_v^{*(k)}$ in the product $S_\lambda{S}_u^{*(k)}$. In this way we recover certain structure constants of the multiplication of dual $k$-Schur functions since when $\lambda\subseteq (c^r)$ and $c+r=k+1$ we have that $S_\lambda=S_w^{(k)*}$ for some $w$ affine grassmannian. 

\noindent{\bf (II)} We prove combinatorially that the expansion of $K_{[u,v]}$ in terms of Schur functions is positive. This is done in analogy with ~\cite{bsottilemonoid, A_JAMS, assafbsottile}.

In this paper we will cover part {\bf (I)} together with some related work and an explicit embedding of the Schubert vs. Schur problem into the dual $k$-Schur problem. This is done by inclusion of the chains of the grassmannian-Bruhat order into the affine strong Bruhat graph. Then a connected component from the first graph is sent to a connected component in the second graph. This implies that the dual $k$-Schur problem is at least as complex as the Schubert vs. Schur problem. From our point of view, this is a very surprising fact. On the dual $k$-Schur function side, we are multiplying affine grassmannians. In the non-affine case, this should correspond to multiplying Schur functions. Here we show that an arbitrary Schubert multiplied by a Schur embeds in the affine case. 

Part ${\bf (II)}$ will appear in~\cite{benedettib} after~\cite{A_JAMS, assafbsottile} is published. 

One final remark before we get started. The approach in~\cite{bmws} cannot be used directly on the affine weak Bruhat order  to understand the multiplication of $k$-Schur functions. It was erroneously suggested in example~6.9 of~\cite{bmws} that $K_{[u,v]_w}$ defined on an interval $[u,w]_w$ of  the affine weak order expands positively in terms of fundamental quasisymmetric functions using descent. The problem here is that the descent of a chain is not well defined. Equation~(6.1) of~\cite{bmws} is valid only if the descent set of a chain is a unique coarsening of its possible decomposition into increasing components. This is not the case in example~6.9 and
going back to the original definition of $K_{[u,v]_w}$ is necessary. The (symmetric) quasisymmetric function $K_{[u,v]_w}$ obtained this way has been rediscovered by Postnikov 
in~\cite{postnikov}. They are not positive in general, but when restricted to the coefficient of a Schur function $S_\lambda$ where $\lambda$ is contained in the fixed rectangle $R$, then the constant is positive and equals to the Gromov-Witten invariants. 
However, $K_{[u,v]_w}$ is not positive when expanded in terms of fundamental quasisymmetric functions and the techniques of~\cite{A_JAMS,assafbsottile} cannot be adapted. Nevertheless, here we show that the affine strong Bruhat graph behaves well.

The paper is organized as follows. In Sections~\ref{sc:Schubert} and ~\ref{sc:kSchur} we recall some background about Schubert polynomials and $k$-Schur functions, respectively. In Section~\ref{sc:strongbruhat} we study the affine strong  Bruhat graph and introduce the main relations satisfied by saturated chains in this order. Also,we introduce the quasi-symmetric function $K_{[u,v]}$. Finally, Section~\ref{sc:SSimbedding} is dedicated to the inclusion of the chains of the grassmannian-Bruhat order.

\section{Schubert Polynomials}\label{sc:Schubert}

One of our main goals is to show that the Schubert vs. Schur problem is embedded in the problem of
multiplying dual $k$-Schur functions, as explained in the introduction. We thus recall a few results from~\cite{bsottilemonoid,bsottilehopf,bsottileskew,bmws}.

Let $u\in {\mathcal S_{\infty}}:=\bigcup_{n\ge0} {\mathcal S}_n$ be an infinite permutation where all but a finite number of positive integers are fixed. Non-affine Schubert polynomials $\mathfrak S_u$ are indexed by such permutations ~\cite{LS82a,Macdonald91}.
These polynomials form a homogenuous basis of the polynomial ring ${\mathbb Z}[x_1,x_2,\ldots]$ in countably many variables. The coefficients $c_{u,v}^w$ in
\begin{equation}\label{eq:genschub}
  {\mathfrak S}_u {\mathfrak S}_v = \sum_{v} c_{u,v}^w {\mathfrak S}_w ,
\end{equation}
are known to be positive. 

\subsection{$r$-Bruhat order and Pieri operators.}\label{sec:rBruhat}
As shown in example~6.2 of~\cite{bmws} (see also~\cite{bsottileskew}), we can encode some of the coefficients in (\ref{eq:genschub}) with a quasisymmetric function as follows.
Let
$\ell(w)$ be the length of a permutation $w\in {\mathcal S}_\infty$.
We define the $r$-Bruhat order $<_r$ by its covers.
Given permutations $u,w\in{\mathcal S}_\infty$, we say that
$u\lessdot_r w$ if $\ell(u)+1=\ell(w)$ and
$u^{-1}w=(i,j)$, where $(i,j)$ is a reflection with $i\leq r<j$.
When $u\lessdot_r w$, we write $wu^{-1}=(a,b)$ with $a<b$ and label the
cover $u\lessdot_r w$ in the $r$-Bruhat order with the integer $b$.

We enumerate chains  in the $r$-Bruhat order according to the descents in their sequence 
of labels of the edges. More precisely,
we use the {\sl descent Pieri operator} 
\begin{equation}\label{eq:PieriOpSchubert}
   x.\h_k\ :=\ \sum_\omega {\rm end}(\omega),
\end{equation}
where the sum is over all chains $\omega$ of length $k$  in the $r$-Bruhat order starting at $x\in{\mathcal S}_\infty$,
 $$
  \omega \ :\  x  \ \stackrel{b_1}{\longrightarrow}\
               x_1\ \stackrel{b_2}{\longrightarrow}\  \cdots
                    \stackrel{b_k}{\longrightarrow}\  x_k\ =:\ {\rm end}(\omega)\,,
 $$
with no descents, that is $b_1\leq b_2\leq\cdots\leq b_k$.  
\noindent Let $\langle{\cdot,\,\cdot}\rangle$ be the bilinear form on ${\mathbb Z} {\mathcal S}_\infty$ induced by the
Kronecker delta function on the elements of ${\mathcal S}_\infty$. Given $u\le_r w$, let $n=\ell(w)-\ell(u)$ be the rank of the interval $[u,w]_r$ and let
\begin{equation}\label{eq:KschubM}
  K_{[u,w]_r}\ =\ \sum_{ \alpha\models n} \langle u.\h_{\alpha_1}...\h_{\alpha_k},w\rangle M_\alpha
\end{equation}
summing over all compositions $\alpha=(\alpha_1,\ldots,\alpha_k)$ of $n$, where
  $$M_\alpha = \sum_{i_1<i_2<\cdots < i_k} x_{i_1}^{\alpha_1} x_{i_2}^{\alpha_2} \cdots x_{i_k}^{\alpha_k}$$
 is the monomial quasisymmetric function indexed by $\alpha$ (see~\cite{aguiarbsottile,bmws}). 
 
\noindent Now, given a saturated chain $\omega$ in the interval $[u,w]_r$ with labels $b_1,b_2,\ldots,b_n$, we let $D(\omega)=(d_1,d_2,\ldots,d_s)$ denote the unique composition of $n$ such that $b_i>b_{i+1}$ exactly in position $i\in\{d_1,d_1+d_2,\ldots,d_1+d_2+\cdots+d_{s-1}\}$. The chain $\omega$ contributes to the coefficient of $M_\alpha$ if and only if $\alpha \le D(\omega)$ under refinement. We thus have
\begin{equation}\label{eq:KschubF}
  K_{[u,w]_r}\ =\ \sum_{\omega\in [u,w]_r} F_{D(\omega)}.
\end{equation}
where $F_\beta$ denotes the fundamental quasisymmetric function for a composition $\beta$.

The descent Pieri operators on this labelled poset are symmetric as
$\h_m$  models the action of the Schur polynomial $h_m(x_1,\ldots ,x_r)$ on
the basis of Schubert classes (indexed by ${\S}_\infty$) in the
cohomology of the flag manifold $SL(n,\mathbb{C})/B$. The quasisymetric function $K_{[u,w]_r}$ 
is then a symmetric function and we can expand it in terms of Schur functions $S_\lambda$.

\begin{proposition}[\cite{bsottileskew}] \label{prop:BS}
\begin{equation}\label{eq:KschubS}
  K_{[u,w]_r}\ =\ \sum_\lambda c^w_{u,(\lambda,r)}\, S_\lambda
\end{equation}
where $c^w_{u,(\lambda,r)}$ is the coefficient of the Schubert polynomial
${\mathfrak S}_w$ in the product\break
${\mathfrak S}_u\cdot S_\lambda(x_1,\ldots,x_r)$.
\end{proposition}

Geometry shows that these coefficients $c^w_{u,(\lambda,k)}$ are non-negative.
To our knowledge, the work in ~\cite{assafbsottile} is the first combinatorial proof of this fact.

\noindent Let us recall the combinatorial analysis in~\cite{bsottilemonoid} to study chains in the $r$-Bruhat order. By definition, a saturated chain in $[u,w]_r$ of the form
 $$
  \omega \ :\  u=u_0  \ \stackrel{b_1}{\longrightarrow}\
               u_1\ \stackrel{b_2}{\longrightarrow}\  \cdots
                    \stackrel{b_n}{\longrightarrow}\  u_n=w\,,
 $$
 is completely characterized by the sequence of transpositions $(a_1,b_1),(a_2,b_2), \ldots (a_n,b_n)$ where
  $(a_i,b_i)u_{i-1}=u_i$. Let $\u{\a\b}$ denote the operator on ${\mathbb Z} \S_\infty$ defined by 
\begin{equation}\label{eq:operat}
\begin{array}{rcl}
\u{\a\b} \colon\  {\mathbb Z}\S_\infty&\longrightarrow& \quad{\mathbb Z}\S_\infty,\\
u\quad&\longmapsto&\ \ \rule{0pt}{28pt} \left\{\begin{array}{ll}
(\a\,\,\,\b)u   
&\mbox{ if } u\lessdot_r (\a,\b)u,
\medskip       
\\      0& \mbox{ otherwise.}\end{array}\right.
\end{array} 
\end{equation}
We have shown in~\cite{bsottilemonoid} that these operators satisfy  the following relations:
\begin{equation} \label{eq:relschub}
\begin{array}{clrclll}
(1)&&\u{\b\c}\u{\c\d}\u{\a\c}&\equiv&\u{\b\d}\u{\a\b}\u{\b\c},\hfill&&
        \hbox{if $\a<\b<\c<\d$},\hfill\\
(2)\hfill&& \hfill \u{\a\c}\u{\c\d}\u{\b\c}&\equiv&
\u{\b\c}\u{\a\b}\u{\b\d},\hfill&&
        \hbox{if $\a<\b<\c<\d$},\hfill\\
(3)\hfill&&  \hfill \u{\a\b}\u{\c\d}&\equiv&\u{\c\d}\u{\a\b}, \hfill&&
        \hbox{if $\b<\c$ or $\a<\c<\d<\b$},\hfill\\
(4)\hfill&&  \hfill   \u{\a\c}\u{\b\d}&\equiv& \u{\b\d}\u{\a\c}\ 
\equiv \  {\bf 0},\hfill&&
        \hbox{if $\a\le \b<\c\le\d$},\hfill\\
(5)\hfill&&  \hfill  \u{\b\c}\u{\a\b}\u{\b\c}
&\equiv&\u{\a\b}\u{\b\c}\u{\a\b}\ \equiv\ {\bf 0},\hfill&&
        \hbox{if $\a<\b<\c$}.
\end{array}  
 \end{equation}
The ${\bf 0}$ in relations (4) and (5) means that no chain in any $r$-Bruhat order can contain such a sequence of transpositions.
On the other hand, relations (1), (2) and (3) are complete and transitively connect any two chains in a given interval $[u,w]_r$.
It is also important to notice that the relations are independent of $r$. This is a fact noticed in~\cite{bsottileschub}: a nonempty interval $[u,w]_r$ in the $r$-Bruhat order is isomorphic to a nonempty interval $[x,y]_{r'}$ in an $r'$-Bruhat order as long as $wu^{-1}=yx^{-1}$. This implies several identities among the structure constants. 

When we write a sequence of operators
$[\u{a_nb_n},\ldots ,\u{a_2b_2},\u{a_1b_1}]$ (or  shortly\break $\u{a_nb_n}\cdots\u{a_2b_2} \u{a_1b_1}$), if nonzero, it corresponds to a unique chain in some nonempty interval $[u,w]_r$ for some $r$ and $w^{-1}u=(a_n,b_n)\cdots(a_1,b_1)$. 
To compute the quasisymmetric function $K_{[u,w]_r}$ as in 
equation~(\ref{eq:KschubF}), it suffices to generate one chain in $[u,w]_r$ and we can obtain the other ones using relations~(1), (2) and (3) above. 

Given any $\zeta\in\S_\infty$ we produce a chain in a nonempty interval $[u,w]_r$ 
as follows. Let $up(\zeta)=\{a:\zeta^{-1}(a)<a\}$. This is a finite set and we can set $r=|up(\zeta)|$. To construct $w$, we sort the elements in $up(\zeta)=\{i_1<i_2<\cdots<i_r\}$ and its complement $up^c(\zeta)={\mathbb Z}_{>0}\setminus up(\zeta)=\{j_1<j_2<\ldots\}$. Next, we put $w=[i_1,i_2,\ldots,i_r,j_1,j_2,\ldots]\in\S_\infty$ and then we let $u=\zeta^{-1}w$. Notice that $u,w$ and $r$ constructed this way depend on $\zeta$. From~\cite{bsottileschub, bsottilemonoid}, we have that $[u,w]_r$ is non-empty and now we want to construct a chain in $[u,w]_r$. This is done recursively as follows: let
\begin{align*}
  a_1 &= u(i_1) \text{ where } i_1=\max\{i\le r:u(i)<w(i)\} \text{ \;\;\; and  }\\
  b_1 &= u(j_1) \text{ where } j_1=\min\{j> r:u(j)>u(i_1)\ge w(j)\}
\end{align*}
then $\u{a_nb_n}\cdots\u{a_2b_2} \u{a_1b_1}$ is a chain in $[u,w]_r$ for any chain $\u{a_nb_n}\cdots \u{a_2b_2}$ in $[(a_1, b_1)u,w]_r$. 

\begin{example} \label{ex:schubert}
Consider $\zeta=[3, 6, 2, 5, 4, 1,...]$ where all other values are fixed. We have that $up(\zeta)=\{3,5,6\}$ and $up^c(\zeta)=\{1,2,4,...\}$. In this case, $r=3$, $w=[3,5,6,1,2,4,...]$ and $u=[1,4,2,6,3,5,...]$. The recursive procedure above produce the chain $\u{23}\u{12}\u{45}\u{26}$ in $[u,v]_3$. We get all other chains by using the relations~(\ref{eq:relschub}):
\begin{equation} \label{eq:axampleschub}
\begin{array}{cccc}
\u{23}\u{12}\u{45}\u{26},&
\u{23}\u{12}\u{26}\u{45},&
\u{23}\u{45}\u{12}\u{26},&
\u{45}\u{23}\u{12}\u{26},\cr
\u{45}\u{13}\u{36}\u{23},&
\u{13}\u{45}\u{36}\u{23},&
\u{13}\u{36}\u{45}\u{23},&
\u{13}\u{36}\u{23}\u{45}.
\end{array}
 \end{equation}
The interval obtained in this case is
$$
 \raise -50pt
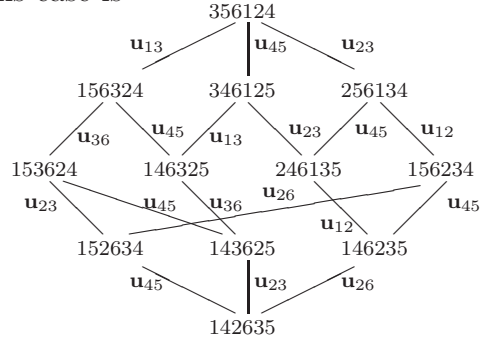
\hbox{ \begin{picture}(200,120)
      \put(100,0){$\scriptstyle 142635$}
      \put(50,30){$\scriptstyle 152634$}
      \put(100,30){$\scriptstyle 143625$}
      \put(150,30){$\scriptstyle 146235$}
      \put(25,60){$\scriptstyle 153624$}
      \put(75,60){$\scriptstyle 146325$}
      \put(125,60){$\scriptstyle 246135$}
      \put(175,60){$\scriptstyle 156234$}
      \put(50,90){$\scriptstyle 156324$}
      \put(100,90){$\scriptstyle 346125$}
      \put(150,90){$\scriptstyle 256134$}
      \put(100,120){$\scriptstyle 356124$}
      \put(120,8){\line(2,1){35}} \put(110,8){\line(-2,1){35}}  \put(115,8){\line(0,1){20}}
      \put(70,18){$\scriptstyle \u{45}$}  \put(117,18){$\scriptstyle \u{23}$}  \put(150,18){$\scriptstyle \u{26}$}
      \put(170,38){\line(1,1){20}} \put(160,38){\line(-1,1){20}}  
      \put(142,40){$\scriptstyle \u{12}$}  \put(190,48){$\scriptstyle \u{45}$}  
      \put(110,38){\line(-1,1){20}} \put(105,38){\line(-3,1){60}}  
      \put(100,48){$\scriptstyle \u{36}$}  \put(75,48){$\scriptstyle \u{45}$}  
      \put(60,38){\line(-1,1){20}} \put(70,38){\line(6,1){110}}  
      \put(30,48){$\scriptstyle \u{23}$}  \put(120,52){$\scriptstyle \u{26}$}  
      \put(140,68){\line(1,1){20}} \put(135,68){\line(-1,1){20}}  
      \put(130,76){$\scriptstyle \u{23}$}  \put(155,76){$\scriptstyle \u{45}$}  
      \put(90,68){\line(1,1){20}} \put(85,68){\line(-1,1){20}}  
      \put(78,76){$\scriptstyle \u{45}$}  \put(100,73){$\scriptstyle \u{13}$}  
      \put(40,68){\line(1,1){20}} 
        \put(50,73){$\scriptstyle \u{36}$}  
      \put(185,68){\line(-1,1){20}}  
        \put(180,76){$\scriptstyle \u{12}$} 
      \put(110,118){\line(-2,-1){35}} \put(120,118){\line(2,-1){35}}  \put(115,118){\line(0,-1){20}}
      \put(70,108){$\scriptstyle \u{13}$}  \put(117,108){$\scriptstyle \u{45}$}  \put(150,108){$\scriptstyle \u{23}$}
     \end{picture}} 
$$
Using the chains in~(\ref{eq:axampleschub}) and equation~(\ref{eq:KschubF}) we can compute the quasisymmetric function associated to this interval and we get
$$
 \begin{array} {rcl}
   K_{[142635,356124]_3} &=& F_{13} + F_{121} + F_{22} + F_{112} + F_{121} + F_{31} + F_{211} + F_{22}\cr
                                          &=& S_{31} + S_{22} + S_{211}.
 \end{array}$$
\end{example}

\null

Notice that the functions $K_{[u,w]_r}$ encode the nonzero connected components of the given interval under the relations~(\ref{eq:relschub}).

The combinatorial proof of the positivity of the $c_{u,(\lambda,r)}^w$ coefficients exposed in ~\cite{assafbsottile} uses the techniques given in~\cite{A_JAMS} in the sense that the construction of a weak dual graph on the chains of $[u,v]_k$ is done by means of a refinement of the relations~(\ref{eq:relschub}). In other words, to go from equation~(\ref{eq:KschubF}) to equation~(\ref{eq:KschubS}) one needs to understand fully the combinatorics of the chains in $[u,w]_r$, as we briefly reviewed here, and then define natural
dual knuth operations on the chains, along the lines of~\cite{assafbsottile}. 

In Section~\ref{sc:SSimbedding} we will show that the connected components
of the chains for the $r$-Bruhat order where $r$ is arbitrary, embed as a connected component of the corresponding theory for the $0$-grassmannian in the affine strong Bruhat graph governing the multiplication of dual $k$-Schur functions.

\section{$k$-Schur Functions and affine Grassmannians.}\label{sc:kSchur}

The $k$-Schur functions were originally defined combinatorially in terms of
$k$-atoms, and conjecturally provide a positive decomposition of the Macdonald
polynomials~\cite{LLM03}. These functions have several definitions and it is conjectural that they are equivalent (see~\cite{LLMS}). In this paper we will adopt the definition given by the $k$-Pieri rule
and $k$-tableaus (see~\cite{lapointemorseDEF,LLMS}) since this gives us a relation with the homology and cohomology of the affine grassmannians and therefore, we get positivity in their structure constants.

Different objects index $k$-Schur functions: $0$-grassmannian permutations, $k+1$-cores, $k$-bounded partitions. Originally (as in~\cite{LLM03}), $k$-Schur functions were indexed by $k$-bounded partitions $\lambda=(\lambda_1,\lambda_2,\ldots,\lambda_\ell)$ where $\lambda_1\le k$. These partitions are in bijection with $k+1$-cores (see ~\cite{LM1}). By definition, $k+1$-cores are integer partitions $\mu=(\mu_1,\mu_2,\ldots,\mu_m)$ with no hook of lenght $k+1$. To close the loop,  in~\cite{bjornerbrenti} it is shown that $k+1$-cores are in bijection with $0-$grassmannian permutations in the affine symmetric group (see also~\cite{BBTZ, LLMS}).

\subsection{Affine Grassmannians and the affine weak order.}
The affine symmetric group $W$ is generated by reflections $s_i$ for $i \in \{0,1,\ldots,k\}$, subject to the relations: 
 $$		s_i^2 = 1 ;\qquad
		s_is_{i+1}s_i = s_{i+1}s_is_{i+1};\qquad
		s_is_j = s_js_i  \ \textrm{ if } i-j \neq \pm 1,
 $$
 where $i-j$ and $i+1$ are understood to be taken modulo $k+1$.
 Let $w\in W$ and  denote its length by $\ell(w)$, given by the minimal number of generators needed to write a reduced expression for $w$. We let $W_0$ denote the parabolic subgroup obtained from $W$ by removing the generator $s_0$. This is naturally isomorphic to the symmetric group $\S_{k+1}$.
For more details on affine symmetric group see \cite{bjornerbrenti}.

Let $u\in W$ be an affine permutation. This permutation can be represented using window notation. That is, $u$ can be seen as a bijection from $\mathbb Z$ to $\mathbb Z$, so that if $u_i$ is the image of the integer $i$ under $u$, then it can be seen as a sequence:
$$
u=\cdots | u_{-k}\; \cdots\; u_{-1}\; u_{0}\underbrace{| u_1\; u_2\; \cdots\; u_{k+1} |}_{\text{main window}}u_{k+2}\; u_{k+3}\; \cdots\; u_{2k+2}|\cdots
$$
Moreover, $u$ satisfies the property that  $u_{i+k+1}=u_i+k+1$ for all $i$, and the sum of the entries in the main window $u_1+u_2+ \cdots+ u_{k+1}={{k+2}\choose{2}}$. Notice that in view of the first property, $u$ is completely determined by the entries in the main window. 
In this notation, the generator $u=s_i$ is the permutation such that $u_{i+m(k+1)}=i+1+m(k+1)$ and $u_{i+1+m(k+1)}=i+m(k+1)$ for all $m$, and $u_j=j$ for all other values. The multiplication $uw$ of permutations $u,w$ in $W$ is the usual composition given by $(uw)_i = u_{w_i}$. In view of this, the parabolic subgroup $W_0$ corresponds to the $u\in W$ such that the numbers $\{1,2,\ldots,k+1\}$ appear in the main window. 

Now, let $W^0$ denote the set of minimal length coset representatives of $W/W_0$. In this paper we take right coset representatives, although left coset representatives could be taken also. The set of permutations in $W^0$ are the \emph{affine grassmannian permutations} of $W$, or $0$-grassmannians for short. 

\begin{definition}\label{grassmannian}
The {\sl affine $0$-grassmannian} $W^0$ are  the permutations $u\in W$ such that the numbers $1,2,\ldots,k+1$ appear from left to right in the sequence $u$.\end{definition}

\begin{example}\label{5core}
Let $k=4$ and 
$$
u=\cdot\cdot\cdot | \bar{3}\;\bar{2}\;1\;\bar{5}\;\bar{1}\underbrace{| 2\;3\;6\;\bar{0}\;4|}_{\text{main window}}7\;8\;11\;5\;9|\cdot\cdot\cdot 
$$
where $\bar{i}$ stands for $-i$. By convention we say that $0$ is negative.
This permutation $u$ is $0$-grassmannian and it corresponds to the $5$-core $\mu=(4,1,1)$. The correspondence
is easy to see from the window notation. We just need to read the sequence of entries of $u$, drawing a vertical step down for each negative entry,
and an horizontal step right for each positive entry. The result is the diagram of $\mu$:
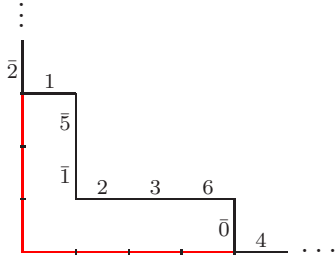
$$
 \raise -0pt\hbox{ \begin{picture}(100,100)
 \rouge{   \put(0,0) {\line(1,0){80}}
               \put(0,0) {\line(0,1){60}} }
 \put(-1,20) {\line(1,0){2}}    \put(-1,40) {\line(1,0){2}}     \put(-1,60) {\line(1,0){2}}   
 \put(20,-1) {\line(0,1){2}}    \put(40,-1) {\line(0,1){2}}     \put(60,-1) {\line(0,1){2}}    \put(80,-1) {\line(0,1){2}}
 \put(0,60) {\line(0,1){20}}   \put(0,60) {\line(1,0){20}}  \put(20,60) {\line(0,-1){40}} \put(20,20) {\line(1,0){60}}
  \put(80,20) {\line(0,-1){20}} \put(80,00) {\line(1,0){20}}
  \put(-2,85){$\vdots$}    \put(105,0){$\ldots$}
  \put(-6,66){$\scriptstyle \bar{2}$}  
  \put(8,62){$\scriptstyle {1}$}  
  \put(14,46){$\scriptstyle \bar{5}$}    \put(14,26){$\scriptstyle \bar{1}$}  
  \put(28,22){$\scriptstyle {2}$}    \put(48,22){$\scriptstyle {3}$}    \put(68,22){$\scriptstyle {6}$}  
  \put(74,6){$\scriptstyle \bar{0}$}  
  \put(88,2){$\scriptstyle {4}$}     
      \end{picture}} 
$$
\end{example}

\subsection{$k$-Schur functions.}
As previously  mentioned, $0$-grassmannian permutations index $k$-Schur functions, which we will denote by $S_u^{(k)}$ for some $u\in W^0$. 

Given $u\in W$, we say that $u \lessdot_w us_i$ is a cover for the weak order if $\ell(us_i)=\ell(u)+1$ and we label this cover by $i$. The weak order on $W$ is the transitive closure of these covers. The Pieri rule for $k$-Schur functions is described by certain chains in the weak order of $W$ restricted to $W^0$. This result is given in ~\cite{lapointemorseDEF,Lam,LLMS}. On the other hand, this  same rule is satisfied by the Schubert grassmannian for the affine symmetric group~\cite{Lam}. 

\noindent Here, we describe the Pieri rule as follows. 
A saturated chain $\omega$ of length $m$ in the weak order with end point ${\rm end}(\omega)$, gives us a sequence of labels $(i_1,i_2,\ldots,i_m)$. We say that the sequence $(i_1,i_2,\ldots,i_m)$ is cyclically increasing if $i_1,i_2,\ldots,i_m$ lies clockwise
on a clock with hours $0,1,\ldots,k$ and $\min\big\{ j : 0\le j\le k; \ j\notin \{i_1,i_2,\ldots,i_m\}\big\}$ lies between $i_m$ and $i_1$.
In particular we must have $1\le m\le k$.
Now, to express the Pieri rule, we first remark that for $1\le m\le k$, the homogeneous symmetric function $h_m $ corresponds to the $k$-Schur function $S_{v(m)}^{(k)}$ where $v(m)$ is a $0$-grassmannian whose main window is given by $|2\;\cdots\;m\;\bar{0}\;m+1\;\cdots\;k\;k+2|$.
Then, the multiplication of a $k$-Schur function $S_u^{(k)}$ by a homogeneous symmetric function $h_m$   
is given by
\begin{equation}\label{eq:PierikSchur}
   S_u^{(k)} h_m:=\ \sum_{\omega \text{ cyclically increasing }}S_{{\rm end}(\omega)}^{(k)},
\end{equation}
where $\omega$ has length exactly $m$.

\noindent Iterating equation (\ref{eq:PierikSchur}) one can easily see that
\begin{equation}\label{eq:hkschur}
 h_\lambda = \ \sum_{u}  {\rm K}_{\lambda,u}S_{u}^{(k)}
 \end{equation}
is a triangular relation~\cite{lapointemorseDEF}. One way to define $k$-Schur functions is to start with equation~(\ref{eq:PierikSchur}) as a rule, and define them as follows.
\begin{definition}\label{def:kschu}
 The {\sl $k$-Schur functions} are the unique symmetric funtions $S_{u}^{(k)}$ obtained by inverting the matrix $[{\rm K}_{\lambda,u}]$ obtained from (\ref{eq:hkschur}) above.
\end{definition}

 It is clear that we can define a Pieri operator like equation~(\ref{eq:PieriOpSchubert})
using the notion of a cyclically increasing chain. Using equation~(\ref{eq:KschubM}), this allows us to define a function $K_{[u,w]_w}$ for any interval in the weak order of $W$. 

\begin{example} \label{ex:kschur}
Let $k=2$ and $u=|\bar{0}\; 2\; 4|$. We consider the interval $[u,w]_w$ in the weak order where $w=|\bar{3}\;4\;5|$. This interval is a single chain 
 $$
 u =|\bar{0}\; 2\; 4|\ \stackrel{1}{\longrightarrow}\
               |2\; \bar{0}\; 4|\ \stackrel{2}{\longrightarrow}\  |2 \; 4\; \bar{0}| 
                    \stackrel{0}{\longrightarrow}\  |\bar{3}\;4\;5|=w\,.
 $$
In this case, we remark that $\langle u.\h_{1}\h_{1}\h_{1},w\rangle=\langle u.\h_{2}\h_{1},w\rangle=\langle u.\h_{1}\h_{2},w\rangle=1$
are the only nonzero entries in~(\ref{eq:KschubM}) and we get
$$
 \begin{array} {rcl}
   K_{[u,w]_w} &=& M_{111}+M_{21}+M_{12}\cr
                       &=& F_{12} + F_{21} - F_{111} \cr
                       &=& S_{21} - S_{111}.
 \end{array}$$
\end{example}

\noindent This small example shows some of the behavior of the (quasi)symmetric function $K_{[u,w]_w}$ for the weak order of $W$. In general, it is not $F$-positive nor Schur positive. Although, these functions contain some information about the structure constants, it is not enough to fully understand them combinatorially, in particular, these functions lack some of the properties needed to use the theory developed in ~\cite{A_JAMS}. These functions were first defined in ~\cite{bmws} in terms of the $M$-basis, but the definition given there in terms of the $F$-basis is wrong. Later on, Postnikov rediscovered them in ~\cite{postnikov} with more combinatorics involved, even though their combinatorial expansion in terms of Schur functions is still open.

\subsection{Dual $k$-Schur functions.}

Let $\Lambda = \mathbb Z[h_1,h_2,\dots]$ be the Hopf algebra of symmetric functions (see~\cite{macdonald} for more details on symmetric functions). The space of $k$-Schur functions $\Lambda_{(k)}$ can be seen as a subalgebra of $\Lambda$ spanned by $ {\mathbb Z}[h_1,h_2,\ldots,h_k]$. In fact, it is a Hopf subalgebra whose comultiplication defined in the homogeneous basis is
given by
 $$\Delta(h_m) =\sum_{i=0}^m h_i\otimes h_{m-i} $$
 and extended algebraically.
 The degree map is given by $\deg(h_m)=m$. 
 The space $\Lambda$ is a self dual Hopf algebra where the Schur functions $S_\lambda$ form a self dual basis under the pairing $\langle h_\lambda,m_\mu \rangle=\delta_{\lambda,\mu}$ where the $m_\lambda$ denote the monomial symmetric functions. 
 
\noindent Now, by the previous paragraph we have the inclusion $\Lambda_{(k)}\hookrightarrow \Lambda$, which turns into a projection $\Lambda\to\!\!\!\!\!\to\Lambda^{(k)}$ when passing to the dual space, where $\Lambda^{(k)}=\Lambda_{(k)}^*$ is the graded dual of $\Lambda_{(k)}$. It can be checked that the kernel of this projection is the linear span of $\{m_\lambda:\lambda_1>k\}$, hence
  $$\Lambda^{(k)}\  \cong\  \Lambda \big/ \langle m_\lambda:\lambda_1>k \rangle\,. $$
The graded dual basis to $S_u^{(k)}$ will be denoted here by ${\frak S}_u^{(k)}=S_u^{(k)*}$ which are also known as the affine Stanley symmetric functions. The multiplication of the dual $k$-Schur ${\frak S}_u^{(k)}$
is described in terms of the affine Bruhat graph as we will see in the next section. 

\section{Affine Bruhat Graph}\label{sc:strongbruhat}

\subsection{Affine Bruhat order.}

Let $t_{a,b}$ be the transposition in $W$ such that for all $m\in\mathbb Z$, permutes $a+m(k+1)$ and $b+m(k+1)$ where $b-a\leq k$.
The \emph{affine Bruhat order}  is given by its covering relation. Namely, for $u\in W$, we  have $u\lessdot ut_{a,b}$ is a cover in the affine Bruhat order if $\ell(ut_{a,b})=\ell(u)+1$.

\begin{proposition}[see \cite{bjornerbrenti}]\label{prop:bruhatcover}
 For $u\in W$ and $b-a\le k$, we have that $u\lessdot ut_{a,b}$ is a cover in the \emph{Bruhat order} if and only if $u(a)<u(b)$ and for all $a<i<b$ we have $u(i)<u(a)$ or $u(i)>u(b)$. 
\end{proposition}

\noindent Notice that if $a'=a+m(k+1)$ and $b'=b+m(k+1)$ then $t_{a',b'}=t_{a,b}$, therefore, many different choices of $a$ and $b$ give the same covering as long as they satisfy the conditions of the proposition. 

\subsection{Affine $0$-Bruhat graph.} 

The affine $0$-Bruhat order arises as a suborder of the Bruhat order. We define it by its covers. For $u\in W$,  a covering $u\lessdot_0 ut_{a,b}$ is encoded by transposition $t_{a,b}$ satisfying proposition \ref{prop:bruhatcover} and also $u(a)\leq 0<u(b)$. As noticed before, a transposition $t_{a',b'}$ satisfying the same conditions as $t_{a,b}$ gives the same affine Bruhat covering relation as long as $a'\equiv a$, $b'\equiv b$ modulo $k+1$. In view of this, we introduce a multigraph instead of a graph for the affine $0$-Bruhat order, since we want to keep track of  the distinct $a,b$ such that $u\lessdot_0 ut_{a,b}$ is an affine $0$-Bruhat covering  for a given $u$.

We then define the following operators in a similar way to the ones defined in equation~(\ref{eq:operat}). For any $b-a\le k+1$, let
\begin{equation}\label{eq:0bruhatop}
\begin{array}{rcl}
\t{\a\b} \colon\  {\mathbb Z}W&\longrightarrow& \quad{\mathbb Z}W,\\
u\quad&\longmapsto&\ \ \rule{0pt}{28pt} \left\{\begin{array}{ll}
ut_{a,b}   
&\mbox{ if } u\lessdot ut_{a,b} \mbox{ and } u(a)\le 0 < u(b)
\medskip       
\\      0& \mbox{ otherwise.}\end{array}\right.
\end{array} 
\end{equation}

We will write these operators as acting on the right: $u\t{\a\b}$. Remark now that if $u\t{\a\b}\ne 0$, then $u\t{\a\b}=u\t{\a',\b'}\ne 0$ for only finitely many values of $m$ with $a'=a+m(k+1)$ and $b'=b+m(k+1)$. To see this, it is enough to notice that there exists $m$ such that $u(a+m(k+1))\geq 0$  and similarly for $b$. 

%

\begin{definition}\label{def:0Bruhatgraph}
The {\sl affine $0$-Bruhat graph} is the directed multigraph with vertices $W$ and a labeled edge $ u  \stackrel{b}{\longrightarrow} u\t{ab} $ for every $u\t{a,b}\ne 0$.
We denote by $[u,w]$ the {\sl set of paths} from $u$ to $w$.
Remark that all such paths will have the same length, namely $\ell(w)-\ell(u)$.
\end{definition}

\begin{example} \label{ex:0Bruhat} We give below the interval $[|\bar{6}\,8\,3\,\bar{1}\,4\,13|,|8\,\bar{6}\,\bar{2}\,9\,13\,\bar{1}|]$ in the affine $0$-Bruhat graph:
$$
\raise 15pt\hbox{ \begin{picture}(300,190)
      \put(40,-2){$\scriptstyle \cdots \bar{8}\,1|\overline{12}\,2\,\bar{3}\,\bar{7}\,\bar{2}\,7
                              \underbrace{\scriptstyle |\bar{6}\,8\,3\,\bar{1}\,4\,13|}_{\text{main}}\bar{0}\,14\cdots$}
      \put(105,8){\line(-2,1){52}}   \put(108,8){\line(-2,1){52}}   \put(111,8){\line(-2,1){52}}  \put(61,16){$\scriptscriptstyle \bar{4};2;8$}  
      \put(114,8){\line(0,1){27}} \put(116,8){\line(0,1){27}} \put(109,20){$\scriptscriptstyle \bar{6}\,\,\,\bar{0}$}  
      \put(120,8){\line(2,1){50}}     \put(150,18){$\scriptscriptstyle 3$}
      \rouge{ \put(130,8){\line(3,1){80}} \put(200,16){$\scriptscriptstyle 5$}}
      \put(30,40){$\scriptstyle |8\,\bar{6}\,3\,\bar{1}\,4\, 13|$}
      \put(90,40){$\scriptstyle |\bar{6}\,8\,3\,\bar{1}\,13\,4|$}
      \put(150,40){$\scriptstyle |\bar{6}\,8\,\bar{2}\,\bar{1}\,9\,13|$}
      \rouge{ \put(210,40){$\scriptstyle |\bar{6}\,8\,3\,4\,\bar{1}\,13|$}}
      \rouge{ \put(148,48){\line(-2,1){78}}   \put(90,80){$\scriptscriptstyle \bar{6}$}}
      \rouge{\put(52,48){\line(4,1){163}}       \put(195,78){$\scriptscriptstyle 5$}}
      \rouge{ \put(204,48){\line(-2,1){78}}   \put(182,52){$\scriptscriptstyle {3}$}}
      \rouge{ \put(100,48){\line(4,1){158}}       \put(240,78){$\scriptscriptstyle 6$}  }   
      \put(18,48){\line(-1,1){40}} \put(20,48){\line(-1,1){40}} \put(-5,52){$\scriptscriptstyle \bar{6}; \bar{0}$} 
      \put(22,48){\line(3,1){118}}       \put(110,80){$\scriptscriptstyle 3$}
      \put(64,48){\line(-2,1){80}}   \put(67,48){\line(-2,1){81}}   \put(70,48){\line(-2,1){81}}    \put(26,72){$\scriptscriptstyle \bar{4};2;8$}  
      \put(75,48){\line(-1,1){40}}    \put(55,70){$\scriptscriptstyle 3$}  
      \put(133,48){\line(-1,1){40}}  \put(135,48){\line(-1,1){40}}  \put(108,58){$\scriptscriptstyle \bar{1}; 5$}  
      \put(140,48){\line(1,3){13}}  \put(142,48){\line(1,3){13}}  \put(144,48){\line(1,3){13}}   \put(155,70){$\scriptscriptstyle \bar{4};2;8$}  
      \put(-40,90){$\scriptstyle |8\,\bar{6}\,3\,\bar{1}\, 13\,4|$}
      \put(20,90){$\scriptstyle |\bar{6}\,8\,\bar{2}\,\bar{1}\,13\,9|$}
      \put(80,90){$\scriptstyle |\bar{6}\,8\,\bar{2}\,9\,\bar{1}\,13|$}
      \put(140,90){$\scriptstyle |8\,\bar{6}\,\bar{2}\,\bar{1}\,9\, 13|$}
       \rouge{    \put(200,90){$\scriptstyle |8\,\bar{6}\,3\,4\,\bar{1}\, 13|$}}
       \rouge{    \put(250,90){$\scriptstyle |\bar{6}\, 8\,3\,4\, 13\, \bar{1}|$}}
        \rouge{ \put(196,98){\line(-2,5){16}} \put(198,98){\line(-2,5){16}}  \put(200,98){\line(-2,5){16}} \put(195,114){$\scriptscriptstyle \bar{6};\bar{0};6$} }
        \rouge{ \put(-30,98){\line(5,1){200}}         \put(100,118){$\scriptscriptstyle 6$}  } 
        \rouge{ \put(168,98){\line(-1,1){40}}         \put(150,107){$\scriptscriptstyle 3$}  } 
        \rouge{ \put(220,98){\line(-4,1){160}}       \put(142,120){$\scriptscriptstyle 3$}  } 
        \rouge{ \put(100,98){\line(-3,1){120}}           \put(-1,136){$\scriptscriptstyle \bar{6}$}}
       \put(-65,98){\line(2,3){27}}    \put(-60,115){$\scriptscriptstyle 3$}  
      \put(-15,98){\line(-1,2){20}} \put(-13,98){\line(-1,2){20}} \put(-11,98){\line(-1,2){20}}        \put(-42,115){$\scriptscriptstyle \bar{4};2;8$}  
      \put(-9,98){\line(3,4){30}} \put(-7,98){\line(3,4){30}}     \put(-8,119){$\scriptscriptstyle \bar{0};6$}  
      \put(40,98){\line(-1,5){8}}   \put(42,98){\line(-1,5){8}}   \put(44,98){\line(-1,5){8}}    \put(20,105){$\scriptscriptstyle \bar{6};\bar{0};6$}  
      \put(46,98){\line(1,1){40}}  \put(48,98){\line(1,1){40}}  \put(50,98){\line(1,1){40}}     \put(68,110){$\scriptscriptstyle \bar{4};2;8$}  
      \put(105,98){\line(-1,5){8}}  \put(107,98){\line(-1,5){8}}    \put(105,110){$\scriptscriptstyle \bar{1};5$}  
      %
      %
      \put(-50,140){$\scriptstyle |8\,\bar{6}\,\bar{2}\,\bar{1}\, 13\,9|$}
      \put(10,140){$\scriptstyle  |\bar{6}\,8\,\bar{2}\,9\,13\,\bar{1}|$}
      \put(70,140){$\scriptstyle |8\,\bar{6}\,\bar{2}\,9\,\bar{1}\,13|$}
     \rouge{ \put(130,140){$\scriptstyle |8\,\bar{6}\,3\,4\, 13\,\bar{1}|$}}
      \put(-33,148){\line(2,1){60}} \put(-30,148){\line(2,1){60}}   
      \put(-16,170){$\scriptscriptstyle \bar{0};6$}  
       \put(31,148){\line(0,1){30}}  \put(33,148){\line(0,1){30}}  \put(35,148){\line(0,1){30}} 
       \put(14,160){$\scriptscriptstyle \bar{4};2;8$}  
      \put(83,148){\line(-3,2){45}}   \put(86,148){\line(-3,2){45}}   \put(89,148){\line(-3,2){45}} 
       \put(45,158){$\scriptscriptstyle \bar{6};\bar{0};6$}  
   \rouge{   \put(138,148){\line(-3,1){90}}   \put(85,170){$\scriptscriptstyle 3$} }
      \put(20,180){$\scriptstyle |8\,\bar{6}\,\bar{2}\,9\,13\,\bar{1} |$}
      \rouge{  \put(152,48){\line(1,3){13}}  \put(154,48){\line(1,3){13}}  \put(156,48){\line(1,3){13}}   \put(138,55){$\scriptscriptstyle \bar{4};2;8$}  }
      \rouge{ \put(150,48){\line(3,2){60}}  \put(153,48){\line(3,2){60}}  \put(156,48){\line(3,2){60}}    \put(202,71){$\scriptscriptstyle \bar{6};\bar{0};6$}  }   
      \rouge{ \put(200,98){\line(-3,2){60}}  \put(203,98){\line(-3,2){60}}  \put(206,98){\line(-3,2){60}}   \put(170,125){$\scriptscriptstyle \bar{4};2;8$}}
      \end{picture}} 
$$
In this example we see that there are three arrows from $u=|\bar{6}\,8\,3\,\bar{1}\,4\,13|$ to $w=|8\,\bar{6}\,3\,\bar{1}\, 13\,4|$.
We have $u\t{\bar{5}\bar{4}}=u\t{12}=u\t{78}=w$ labeled by $\bar 4,2,8$, respectively. Also we have operators that evaluate to 0, namely, $u\t{\overline{11}\,\overline{10}}=0$. 
\end{example}

When restricted to $0$-grassmannian permutations, the affine $0$-Bruhat graph behaves well as shown in the next lemma whose proof (for left coset)  can be consulted in~\cite[Prop. 2.6]{LLMS}. Therefore, we will restrict the affine $0$-Bruhat graph to permutations in $W^0$.

\begin{lemma}\label{lem:0grass}
If $u\t{ab}=w$ and $u\in W^0$, then we have that $w\in W^0$. 
\end{lemma}

Remark that the converse to Lemma~\ref{lem:0grass} is not true. Take $u=|\bar{0}\,2\, 1|$ and $w=|1\,2\,\bar{0}|$. We have that $w\in W^0$ and $w=u\t{13}$, but $u\not\in W^0$.

%
%
%
%

\subsection{Multiplication dual $k$-Schur.}
For dual $k$-Schur functions ${\frak S}^{(k)}_u$, the analogue of the Pieri formula~(\ref{eq:PierikSchur}) is given by
 \begin{equation}\label{eq:PieridualkSchur}
   {\frak S}_u^{(k)} h_m:=\ \sum_{u\t{a_1b_1}\cdots\t{a_mb_m}\ne 0 \atop b_1<b_2<\ldots <b_m}{\frak S}_{u\t{a_1b_1}\cdots\t{a_mb_m}}^{(k)},
\end{equation}
where the sum is over all increasing paths $b_1<b_2<\cdots <b_m$ starting at $u$~\cite{LLMS}.

\null

Since the Pieri formula is encoded by increasing chains in the affine $0$-Bruhat graph restricted to $W^0$, we can define Pieri operators similar to equation~(\ref{eq:PieriOpSchubert})
using increasing chains. 
This allows us to define the functions $K_{[u,w]}$ for any interval in the affine $0$-Bruhat graph restricted to $W^0$.  In contrast with the weak order, where we had cyclically increasing chains, any chain $\omega\in[u,w]$ has a well defined notion of descent. More precisely, for $\omega = \t{a_1b_1}\t{a_2b_2}\cdots\t{a_mb_m}$ we have $D(\omega)=(d_1,d_2,\ldots,d_s)$ denotes the unique composition of $n$ such that $b_i>b_{i+1}$ exactly in position $i\in\{d_1,d_1+d_2,\ldots,d_1+d_2+\cdots+d_{s-1}\}$. As in equation~(\ref{eq:KschubF}) we have
\begin{equation}\label{eq:KdualkschurF}
  K_{[u,w]}\ =\ \sum_{\omega\in [u,v]} F_{D(\omega)}
\end{equation}
and in this case $K_{[u,w]}$ is $F$ positive. 

\begin{theorem}\label{thm:dkschurconstant}
\begin{equation}\label{eq:KdualkschurS}
  K_{[u,w]}\ =\ \sum_\lambda c^w_{u,\lambda}\, S_\lambda
\end{equation}
where $c^w_{u,\lambda}$ is the coefficient of the dual $k$-Schur function
${\mathfrak S}^{(k)}_w$ in the product
${\mathfrak S}^{(k)}_u\cdot S_\lambda$.
\end{theorem}

The proof of this theorem follows from~\cite{bmws}.

\begin{example} \label{ex:0Bruhat2}
Considering the interval $[u,w]=[|\bar{6}\,8\,3\,\bar{1}\,4\,13|,|8\,\bar{6}\,\bar{2}\,9\,13\,\bar{1}|]$ 
we have in example~\ref{ex:0Bruhat}. The total number of chains is  $240$. In this case
  $$K_{[u,w]} =9 F_{1111}+ 30F_{112}+ 51F_{121} + 30F_{13} + 30F_{211}+ 51F_{22} + 30F_{31} + 9F_4 \,,$$
is symmetric and the expansion in term of Schur functions is positive
  $$K_{[u,w]} =9S_4+30S_{31}+21S_{22}+30S_{211}+9S_{1111}\,.$$
The reader is encouraged to use SAGE and see that the coefficients are indeed the structure constants we claim in Theorem~\ref{thm:dkschurconstant}.
\end{example}

\subsection{Relations of the operators $\t{ab}$.}\label{sec:rel0bruhat}
The purpose of this section is to understand some of the relations satisfied by the $\t{ab}$ operators restricted to $W^0$, similar to the work done with Schubert polynomials in ~\cite{assafbsottile,bsottilemonoid}. The main theorem of this section presents the needed relations among these operators. 

These relations  depend  on the following data. For $\t{ab}$ we need to consider $a,b,\overline{a},\overline{b}$ where $\overline{a}$ and $\overline{b}$ are the residue modulo $k+1$ of $a$ and $b$ respectively. Remark that  $\overline{a}\ne \overline{b}$ since $b-a<k+1$. Let $u\in W^0$. Lemma~\ref{lem:0grass} implies  that, if non-zero, $u\t{ab}$ and $u\t{ab}\t{cd}$ are both in $W^0$.  The different relations satisfied by the operators $\t{ab}$ and $\t{cd}$ depend on the relation among $\overline a, \overline b, \overline c, \overline d$. For this reason it is useful to visualize these operators as follows.  

$$
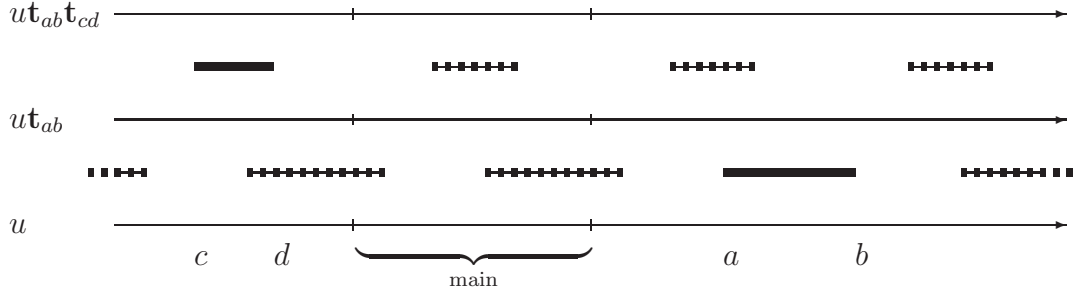
\qquad \raise 10pt\hbox{ \begin{picture}(360,110)
     \put(30,0){$c$} \put(60,0){$d$}  \put(230,0){$a$} \put(280,0){$b$}     
     \put(91,10){$\underbrace{\phantom{|\bar{6}\,8\,3\,\qquad \,\,\,\,\bar{1}\,4\,13|}}_{\text{main}}$}
     \put(-40,12){$u$} \put(0,15){\vector(1,0){360}} \put(90,13){\line(0,1){4}} \put(180,13){\line(0,1){4}}
     \put(-40,52){$u\t{ab}$} \put(0,55){\vector(1,0){360}} \put(90,53){\line(0,1){4}} \put(180,53){\line(0,1){4}}
     \put(-40,92){$u\t{ab}\t{cd}$} \put(0,95){\vector(1,0){360}} \put(90,93){\line(0,1){4}} \put(180,93){\line(0,1){4}}
                      \put(50,35){\line(1,0){50}} \put(140,35){\line(1,0){50}} \put(320,35){\line(1,0){30}} \put(0,35){\line(1,0){10}} 
                      \put(120,75){\line(1,0){30}} \put(210,75){\line(1,0){30}}   \put(300,75){\line(1,0){30}} 
     \linethickness{1mm}
                      \put(230,35){\line(1,0){50}}
                      \multiput(50,35)(5,0){11}{\line(1,0){2}} \multiput(140,35)(5,0){11}{\line(1,0){2}} \multiput(320,35)(5,0){9}{\line(1,0){2}} 
                      \multiput(-10,35)(5,0){5}{\line(1,0){2}}
                      \put(30,75){\line(1,0){30}}
                      \multiput(120,75)(5,0){7}{\line(1,0){2}}  \multiput(210,75)(5,0){7}{\line(1,0){2}}  \multiput(300,75)(5,0){7}{\line(1,0){2}} 
      \end{picture}} 
$$

Above the permutation $u$, the operator $\t{ab}$ is represented by drawing a bold line connecting positions $a,b$ and repeating this pattern to the left and to the right in all positions congruent to $a,b$ modulo $k+1$. Next, to the resulting permutation we apply $\t{cd}$, drawing a bold line connecting positions $c,d$ and repeating that pattern modulo $k+1$. The importance of visualizing not only the bold line but also the dotted ones, relies on the fact that even if in the diagram, the line representing $\t{ab}$ does not intersect the line representing $\t{cd}$, their ``virtual" copies (or dotted copies) might intersect and this will determine the commutation relation satisfied by these operators. Therefore, it will be enough if we consider the pattern produced by these two operators in the main window.

With these definitions in mind we present some of the relations satisfied by the $\mathbf{t}$ operators restricted to $W^0$ (there are less relations if we consider all of $W$).

\smallskip
\noindent {\bf (A)} $\t{ab}\t{cd}\equiv \t{cd}\t{ab}$ \qquad if $\overline{a},\overline{b},\overline{c},\overline{d}$ are distinct. 

\medskip
\noindent {\bf (B1)} $\t{ab}\t{cd}\equiv \t{cd}\t{ab}\equiv 0$ \qquad if ($a<c<b<d$) or ($b=c$ and $d-a>k+1$).

\medskip
\noindent {\bf (B2)} $\t{ab}\t{cd}\equiv 0$ \qquad if ($\overline{a}=\overline{c}$ and $b\le d$) or ($\overline{b}=\overline{d}$ and $c\le a$).

\medskip

\noindent There are more possible zeros than what we present in (B), but we will satisfy ourselves with these ones for now. It will be more important to
identify them in the second part of this work.
Now if the numbers $a,b,c,d$ are not distinct, then we must have $b=c$ or $d=a$. If $b=c$, then $d-a\le k+1$ in view of (B). Similarly if $d=a$ then $b-c\le k+1$. 

\medskip
\noindent {\bf (C1)} $\t{ab}\t{bd}=\t{ab}\t{b-k-1,a}$ \qquad if $d-a=k+1$,  

\smallskip
\noindent {\bf (C2)} $\t{ab}\t{bd}$ \  \ and \ \ $\t{bd}\t{ab}$ \qquad if $d-a<k+1$.

\medskip
Now we look at the cases $\t{ab}\t{cd}$ where $a,b,c,d$ are distinct but some equalities occur between $\overline{a},\overline{b}$ and $\overline{c},\overline{d}$. By symmetry of the relation we will assume that $b<d$ which (excluding (B)) implies that  $a<b<c<d$. 

\medskip
\noindent {\bf (D)} $\t{ab}\t{cd}=\t{d-k-1,c}\t{b-k-1,a}$ \qquad if $\overline{b}=\overline{c}$,\ \  $\overline{d}=\overline{a}$ and $(b-a)+(d-c)=k+1$.

\medskip
All the relations above are {\sl local}. This means that if $\t{ab}\t{cd}=\t{c'd'}\t{a'b'}$, then $|a'-a|$, $|b'-b|$, $|c'-c|$ and $|d'-d|$ are strictly less than $k+1$. For example in (D) we have $|b-k-1-a|$, $|a-b|$, 
$|d-k-1-c|$ and $|c-d|$ which are strictly less than $k+1$. 

\begin{remark}
The relations we care about in this paper and its sequel are all local. There are some relations that are not local:
$$ \t{ab}\t{cd}=\t{a-k-1,b-k-1}\t{cd}=\t{a+k+1,b+k+1}\t{cd},$$
if $c<a<b<d$. The full description of  the relations of the operators $\t{}$ is rather complicated and would take too much space here. It might be an interesting project in the future but at this point we will be satisfied with the given subset. Also,
in his Ph. D. thesis,~\cite{beligan} remarked that intervals $[u,w]_r$ in the $r$-Bruhat order containing chains produced by {\sl nested} operators $\u{ab}\u{cd}$ (i.e. where $c<a<b<d$) are problematic. Schensted insertion and jeu-de-taquin are well behaved as long as the intervals contain no nesting. Here we see that nesting creates even more problems.
\end{remark}

We now consider some more relations of length three:

\medskip
\noindent {\bf (E1)} $\t{bc}\t{cd}\t{ac}\equiv \t{bd}\t{ab}\t{bc}$  \qquad if $a<b<c<d$,

\smallskip
\noindent {\bf (E2)} $\t{ac}\t{cd}\t{bc}\equiv \t{bc}\t{ab}\t{bd}$  \qquad if $a<b<c<d$. 

\medskip\noindent
also we have

\smallskip
\noindent {\bf (F)} $\t{bc}\t{ab}\t{bc}\equiv \t{ab}\t{bc}\t{ab}  \equiv\ {\bf 0}$  \qquad if $a<b<c$ and $c-a<k+1$.  

\medskip

\begin{theorem}
 The relations (A)--(F) above describe relations between $\t{}$-operators in the Strong Bruhat graph.
\end{theorem} 

\proof \null \ \smallskip

\noindent
(A) This relation is clear as the corresponding affine transpositions commute $t_{ab}t_{cd}=t_{cd}t_{ab}$. So if the result is non-zero, it will be non-zero on both sides and equal.

\noindent
(B1) Let us first assume that $a<c<b<d$. We want to show that $u\t{ab}\t{cd}=0$ for all $u\in W^0$.
If $u\t{ab}=0$, then we are done. We thus assume that $w=u\t{ab}\ne 0$. In this case we must have that $u(a)\le 0 < u(b)$ and for all $a<i<b$ we
have $u(i)<u(a)$ or $u(i)>u(b)$. In particular, since $a<c<b$ then $u(c)<u(a)$ or $u(c)>u(b)$. If $w\t{cd}\ne 0$ then $w(c)=u(c)<0$ and thus $u(c)<u(a)$. But also, since $c<b<d$ and since $w(b)=u(a)$ (which is non positive) then $w(b)<w(c)$, or equivalently, $u(a)<u(c)$. This is a contradiction, hence $u\t{ab}\t{cd}=0$. A similar argument allows us to conclude that $u\t{cd}\t{ab}=0$ in this case.

If $b=c$ and $d-a>k+1$, then $a<d-k-1<b<a+k+1<d<b+k+1$. If $u\t{ab}\t{cd}\ne 0$, then we must have $u(a)\le 0<u(b)<u(b+k+1)$ and $0<u(d)$.
We look at the sign of $u(d-k-1)$. If $0<u(d-k-1)$, then since $u\t{ab}\ne 0$ and $a<d-k-1<b$, we must have $u(d-k-1)>u(b)$. This gives $u(d)>u(b+k+1)$.
When we perform $w=u\t{ab}$ we have $w(a+k+1)=u(b+k+1)$ and $w(c)=w(b)=u(a)$. Hence $w(c)\le 0 <w(a+k+1)<w(d)$ a contradiction to $w\t{cd}\ne 0$.
Now if $u(d-k-1)\le 0$, then we must have $u(d-k-1)<u(a)\le 0$. This gives  $0<u(d)<u(a+k+1)\le k+1$ and this is a contradiction to $u\in W^0$ since the entries $1,2,\ldots,k+1$ must appear from left to right. Here we have $u(a+k+1)$ appearing before $u(d)$. We must thus  have $u\t{ab}\t{cd}= 0$
The case $u\t{cd}\t{ab}=0$ is similar.

\noindent
(B2) If $c=a+m(k+1)$ and $b<d$, then ($a<b<c<d$ and $m>0$) or $a=c<b<d$. Assume that $w=u\t{ab}\ne 0$. We have $u(a)\le 0<u(b)=w(a)$.
But then $w(c)=w(a+m(k+1))=w(a)+m(k+1)>0$. This implies that $w\t{cd}=0$.
If $b=d$, then $a=c$ and clearly $\t{ab}\t{ab}=0$. The case when $d=b-m(k+1)$ and $c\leq a$ for $m\leq 0$ is analog.

\noindent
(C1) We have $b-k-1<a<b<a+k+1=d$. If $w=u\t{ab}\ne 0$, then $w(b-k-1)=u(a-k-1)=u(a)-k-1<0<u(b)=w(a)$. Since $wt_{b-k-1,a}=wt_{bd}$, we have that $w\t{bd}\ne 0$ implies $0\ne w\t{b-k-1,a}=w\t{bd}$. The reverse implication is similar.

\noindent
(C2) It suffices to see that for $u=\cdot\cdot|\bar{0}\,2\,4|\cdot\cdot$ we have $u\t{12}\t{23}\ne 0$. On the other hand, we can check that $u\t{12}\t{23} = u\t{\overline 1,\overline 0}\t{13}$ but this is not a local move. Also, it is easy to check that no other moves can be performed on $u$ to obtain $u\t{12}\t{23}$.

\noindent
(D) The conditions imply that $c=b+m(k+1)$ and $d=a+(m+1)(k+1)$ for some $m>0$. We have $b-k-1<a<b<d-k-1<c<d$.
Assume $w=u\t{ab}\ne 0$ so $u(a)\le 0<u(b)<u(b+m(k+1))<u(b+(m+1)(k+1))=w(d)$. For $w\t{cd}\ne 0$ as well we need $w(c)=u(d-k-1)\le 0$.
We also have $u(c)=u(b+m(k+1))>0$. Hence if $u\t{ab}\t{cd}\ne 0$, then $0\ne u\t{d-k-1,c}=w$. Moreover $w(b-k-1)=u(a-k-1)<0<u(b)=w(a)$ and so $0\ne w\t{b-k-1,a}=w\t{cd}$. The argument for the converse is similar.

\noindent
(E1)  Assume $u\t{bc}\t{cd}\t{ac}\ne 0$. Arguing as above we must have $u(a)\le 0$ and $u(b)\le 0<u(d)<u(c)$. We get that $0\ne u\t{bd}\t{ab}\t{bc}=u\t{bc}\t{cd}\t{ac}$.

\noindent
(E2)  The argument is similar to (E1).

\noindent
(F)  If $w=u\t{bc}\t{ab}\ne 0$, then $u(b)\le 0 < u(c)$. But $w(c)=u(b)\le 0$ which implies $w\t{bc}=0$.
The other relation holds in the same way.
\endproof

\begin{remark} \label{rem:addedrel}
\medskip
If we consider the permutation $u$ we can derive more relations of length 2. Let $r=(b-a)+(d-c)$:

\medskip
\noindent {\bf (X1)} $u\t{ab}\t{cd}=u\t{d,c+r}\t{b-r,a}$ \qquad if $r<k+1$, \ $\overline{d}=\overline{a}$, \ $u(c)\le 0$ and $u(d)\le 0$,

\smallskip
\noindent {\bf (X2)} $u\t{ab}\t{cd}=u\t{cd}\t{b-r,b}$ \qquad if $r<k+1$, \ $\overline{d}=\overline{a}$ and $u(d)> 0$,

\smallskip
\noindent {\bf (X3)} $u\t{ab}\t{cd}=u\t{d-r,d}\t{ab}$ \qquad if $r<k+1$, \ $\overline{b}=\overline{c}$ and $u(a+r)\le 0$,

\smallskip
\noindent {\bf (X4)} $u\t{ab}\t{cd}=u\t{d-r,c}\t{b,a+r}$ \qquad if $r<k+1$, \ $\overline{b}=\overline{c}$, \  $u(b)>0$  and $u(a+r)> 0$, 

\smallskip
\noindent {\bf (X5)} $u\t{ab}\t{cd}=u\t{cd}\t{a,b+c-d}$ \qquad if $\overline{b}=\overline{d}$, \  $b-a>d-c$ and $u(d-b+a)> 0$,

\smallskip
\noindent {\bf (X6)} $u\t{ab}\t{cd}=u\t{c,d-b+a}\t{a,b}$ \qquad if $\overline{b}=\overline{d}$, \  $b-a<d-c$ and $u(a)\le 0$. 

\medskip\noindent
In the (X) relations, the conditions we impose on $u$ are minimal to assure that both sides of the equality are non-zero. These conditions are not given by the definition of the operators $\t{ab}$.
For example in (X1), the left hand side is non-zero regardless of the value of $u(d)$ but to guarantee that the right hand side is non-zero, we must have $u(d)\le 0$.
This shows that as operators $\t{ab}\t{cd}\ne \t{d,c+r}\t{b-r,a}$. In the part (II) of our program we will need to study all of the (X) relations. If one considers an interval $[u,w]$ of rank 3 
and computes $K_{[u,w]}$, then by Proposition~\ref{prop:BS} the coefficient of $F_{21}$ and $F_{12}$ must be the same in $K_{[u,w]}$. This means that every time we have a descent
followed by an ascent in a chain, we must have another chain with an ascent followed by a descent. This should be reflected by relations like (X) and could depend on $u$.
The main work of~\cite{benedettib} is first to build a full set of relations of length 3 that pairs every ascent-descent type to a descent-ascent. This cannot be done independently from $u$.
The purpose of this will be to define Dual-Knuth operations on the maximal chains in intervals $[u,w]$ in order to construct dual graphs as in~\cite{A_JAMS}.

\end{remark}

\section{Schubert vs Schur Imbedded Inside Dual $k$-Schur}\label{sc:SSimbedding}

When comparing the relations~(\ref{eq:relschub}) and the ones given in Section~\ref{sec:rel0bruhat} we see that it may be possible to find a homomorphism from the Schubert vs Schur operators $\u{ab}$
to the Dual $k$-Schur operators $\t{a'b'}$. Such a homomorphism vanishes on many chains and this is the expected behavior. The main result of this section is that for any interval $[x,y]_r$ in the $r$-Bruhat order we can find a $k$ and a homomorphism such that every chain of $[x,y]_r$ maps to a non-zero chain in an interval $[u,v]$.

\begin{example} If we compare Example~\ref{ex:schubert} and Example~\ref{ex:0Bruhat}, the map $\u{ab}\mapsto\t{a-3,b-3}$ is a homomorphism that preserves all the chains from the first interval to the second one. This implies that, coefficient-wise, the quasisymmetric function $K_{[142635,356124]_3} $  is smaller than $K_{[u,w]}$. This fact is also implied by noticing that a transposition $\t{a,b}$ could be applied to several windows in a given affine grassmannian permutation $u$, which is not the case, in general, for permutations in the $r$-Bruhat order.
\end{example}

Now, given a non-empty interval $[x,y]_r$ in the $r$-Bruhat order, we want to find integers $k$, $s$ and an explicit interval $[u,v]$ in the strong $0$-Bruhat graph such that the homomorphism $\u{ab}\mapsto\t{a-s,b-s}$  maps the non-zero chains of $[x,y]_r$ to non-zero chains of $[u,v]$. In fact, we only need to assume that we have a non-zero operator $\u{a_nb_n}\cdots\u{a_1b_1}$ and obtain the other ones using the corresponding relations. Then, the interval $[x,y]_r$ is isomorphic to the one described in Section~\ref{sec:rBruhat}. 

\noindent For this purpose, let $\zeta=(a_n, b_n)\cdots(a_1,b_1)$, $up(\zeta)=\{i_1<i_2<\cdots<i_r\}$ and $up^c(\zeta)=\{j_1<j_2<\cdots\}$, then $r=|up(\zeta)|$. As in  Section~\ref{sec:rBruhat} we have that $[x,y]_r$ is nonempty for $y=[i_1,i_2,\ldots,i_r,j_1,j_2,\ldots]$ and $x=\zeta^{-1}y$. 

Let $k$ be such that $\alpha=x(\alpha)=y(\alpha)$ for all $\alpha>k+1$. Such a $k$ exists since $x$ and $y$  have finitely many non-fixed points. Put $x_{\alpha}=x(\alpha)$ and take the permutation $[x_1,x_2,\ldots,x_{k+1}]$. Now, 
we consider the positions $\alpha_1<\cdots<\alpha_\ell<r<\beta_1<\cdots<\beta_t<k+1$ for which there are descents before and after $r$. In other words, where $x_{\alpha_i}>x_{\alpha_{i+1}}$ and $x_{\beta_j}>x_{\beta_{j+1}}$ for $1\leq i\leq\ell -1$ and $1\leq j\leq t-1$. This defines segments
 $$ 1,2,\ldots,\alpha_1;\quad  \cdots\quad {\alpha}_\ell+1,\ldots,r;\quad r+1,\ldots,\beta_1;\quad\cdots\quad \beta_{t}+1,\ldots,k+1.$$
We want to construct a $0$-grassmannian in the $k+1$-affine permutation group $W$ with this information such that in some adjacent $k+1$ positions we have a permutation that has the same patterns as $x^{-1}$. The reason we want to look at the inverse permutation $x^{-1}$ is because the $\u{}$  operators act on the left whereas the $\t{}$ operators act on the right. 

\noindent For this purpose, we first place the values $1,2,\ldots,k+1$ on the $\mathbb Z$-axis as follows. 
$$ \begin{array}{rcl}
1,2,\ldots,k-\beta_t+1&\hbox{ in positions }& x_{\beta_t+1}-t(k\!+\!1),\ldots,x_{k+1}-t(k\!+\!1)\cr
&\cdots&\cr
k-\beta_1+2,\ldots,k-r +1&\hbox{ in positions }& x_{r+1},\ldots,x_{\beta_1}\cr
k-r+2,\ldots,k-\alpha_\ell+1&\hbox{ in positions }& x_{\alpha_\ell+1}+(k\!+\!1),\ldots,x_{r}+(k\!+\!1)\cr
&\cdots&\cr
k-\alpha_1+2,\ldots,k+1&\hbox{ in positions }& x_{1}+(\ell\!+\!1)(k\!+\!1),\ldots,x_{\alpha_1}+(\ell\!+\!1)(k\!+\!1)\cr
  \end{array} $$
  This construction places the values $1,2,\ldots,k+1$ on the $\mathbb Z$-axis from left to right in distinct positions modulo $k+1$.
  We build a permutation $u'$ of $\mathbb Z$ defining it with the relation $u'_{i+m(k+1)}=u'_i+m(k+1)$.
  This may not be a permutation in $W$ as the sum $u'_1+u'_2+\cdots +u'_{k+1}$ may not be ${k+2 \choose 2}$, but a simple shift gives us the desired result, as shown in the next lemma which will be followed by an example to make this construction clearer.
  
  \begin{lemma} Any permutation $u'$ of $\mathbb Z$ such that $u'_{i+m(k+1)}=u'_i+m(k+1)$ and  the values $1,2,\ldots,k+1$ are in distinct positions modulo $k+1$ satisfies
   $$u'_1+u'_2+\cdots +u'_{k+1} = {k+2 \choose 2} - s(k+1)$$
   for some integer $s$.
  \end{lemma}
  
  \begin{proof} Let $w^{-1}=u'$. Since $1,2,\ldots,k+1$ are in distinct positions in $u'$ modulo $k+1$ we 
  have that $w_1+w_2+\cdots+w_{k+1}=1+2+\cdots+(k+1)+s(k+1)$ for some $s\in\mathbb Z$.
  The result follows by inverting the permutation.
  \end{proof}
  
  Notice that each time we shift the values of $u'$ by $1$, like $v_i=u'_{i+1}$ we get that
    $$v_1+v_2+\cdots+v_{k+1}=u'_1+u'_2+\cdots u'_{k+1}+(k+1) = {k+2 \choose 2}+(1-s)(k+1).$$
Hence, if $u'$ is as above and if the entries $1,2,\dots, k+1$ appear from left to right in $u'$, then by defining the permutation $u$ by $u_i=u'_{i+s}$, we get a $0$-affine permutation in $W^0$. 

\begin{example}
Let us take the permutation from Example~\ref{ex:schubert}. 
Let $\zeta=[3, 6, 2, 5, 4, 1,...]$ where all other values are fixed.  We can choose $k+1=6$. We have that $up(\zeta)=\{3,5,6\}$ and $up^c(\zeta)=\{1,2,4,...\}$. In this case, $r=3$, $y=[3,5,6,1,2,4,...]$ and $x=[1,4,2,6,3,5,...]$. The descents in the permutation $x$ are in positions $\alpha = 2$ and $\beta = 4$ so that $\ell=t=1$ and $\alpha < r <\beta$.
With the procedure above, we get 
  $$ \begin{array}{l}
  1=u'(x_5-6)=u'(-3), \qquad 2=u'(x_6-6)=u'(-1);\\
  3=u'(x_4)=u'(6);\\
  4=u'(x_3+6)=u'(8);\\
  5=u'(x_1+12)=u'(13), \qquad 6=u'(x_2+12)=u'(16).
  \end{array}
  $$
 Once we determine the values in the positions above, all other values of $u'$ are determined as follows   
  $$u' = \cdots  | \overline{13} \, \bar{8} \, 1 \, \overline{12} \, 2 \, \bar{3} 
              \underbrace{ | \bar{7} \, \bar{2} \, 7  \, \bar{6} \, 8 \, 3  | }_{\text{main}}
              \bar{1} \, 4 \, 13 \, \bar{0} \, 14 \, 9 | 
              5 \, 10 \, 19 \, 6 \, 20 | \cdots
  $$
the sum of the entries in the main window of $u'$ is $3={7 \choose 2}-3(6)$, hence $s=3$. We see that the entries of $u'$ in the main window $[\bar{7} \, \bar{2} \, 7  \, \bar{6} \, 8 \, 3 ]$  are in the same relative order as $x^{-1}=[1 \, 3 \, 5 \, 2 \, 6 \, 4]$. We also see that the smallest $r=3$ entries of the main window of $u'$ are $\le 0$ and the remaining ones are positive. Now we get $u$ by shifting the positions of $u'$ by $s$:
  $$u = \cdots \overline{13} \, \bar{8} \,1 | \overline{12}\,2\,\bar{3}\,\bar{7}\,\bar{2}\,7
                              \underbrace{ | \bar{6}\,8\,3\,\bar{1}\,4\,13 | }_{\text{main}}\bar{0}\,14 \, 9 \, 
              5 \, 10 \, 19 | 6 \, 20 \cdots
  $$
\end{example}

We remark that by construction, the entries $[u_{1-s},u_{2-s},\ldots,u_{k+1-s}]$ are the same as 
$[u'_1,u'_2,\ldots,u'_{k+1}]$ which in turn are in the same relative order as in $x^{-1}$. Therefore, from the previous paragraph we see that the smallest $r$ entries in $[u_{1-s'},u_{2-s'},\ldots,u_{k+1-s'}]$ are $\leq 0$ and the other entries in that window are positive. This implies that if $x$ is covered by a non-zero permutation given by $\u{ab}x$ where $x^{-1}_{a}\le r <x^{-1}_{b}$,
then we have $u\t{a-s,b-s}$ is a cover in the $0$-Bruhat graph. 
Recursively, we get that

\begin{theorem} Let $[x,y]_r$ be a non-empty interval $[x,y]_r$ in the $r$-Bruhat order and let $u$ and $s$ be as above. For any maximal chain $\u{a_nb_n}\cdots\u{a_1b_1}$ in the interval $[x,y]_r$ 
we have that the chain $\t{a_1-s,b_1-s}\cdots \t{a_n-s,b_n-s}$ is a non-zero maximal chain in the $0$-affine Bruhat graph in $[u,u\t{a_1-s,b_1-s}\cdots \t{a_n-s,b_n-s}]$.
\end{theorem}

This theorem shows our main claim, namely the fact that the Schubert vs Schur problem is imbedded in the dual $k$-Schur problem. In the second part of our program~\cite{benedettib} we will construct dual Knuth operators on the intervals $[u,w]$. Under the morphism above, connected components of certain dual equivalent graphs obtained in~\cite{assafbsottile} are mapped to connected components of the dual equivalent graph of $[u,w]$. This shows in a stronger sense the imbedding above and explains the difficulty of the two problems. This allows us to conclude that solving the dual $k$-Schur problem is harder than the problem of Schubert vs Schur.

\end{document}